\documentclass{article}

\usepackage{amsthm, amsfonts}
\usepackage{graphicx}
\usepackage{mathtools} 
\usepackage{xcolor}
\usepackage{float} 
\usepackage{authblk}

\usepackage[a4paper, total={6in, 8in}]{geometry}

\newtheorem{theorem}{Theorem}
\newtheorem{definition}{Definition}
\usepackage{caption}
\usepackage{subcaption}

\newtheorem{proposition}{Proposition}
\newtheorem{example}{Example}

\usepackage[percent]{overpic}
\usepackage{centernot}

\newcommand\R{\ensuremath{\mathbb R}}

\newcommand\kbsm{\mathcal{K}}
\newcommand\submodule{\mathcal{S}}
\newcommand\basis{\mathcal{B}}

\newcommand\fr{\text{fr}}
\newcommand{\hatLf}{\mathcal{\hat L}_{\Sigma,n}^{\fr}}
\newcommand{\hatBf}{\mathcal{\hat B}_{\Sigma,n}^{\fr}}
\newcommand{\hatKf}{\mathcal{\hat K}_{\Sigma,n}^{\fr}}
\newcommand{\Lf}{\mathcal{L}_{\Sigma,n}^{\fr}}
\newcommand{\Bf}{\mathcal{B}_{\Sigma,n}^{\fr}}
\newcommand{\Kf}{\mathcal{K}_{\Sigma,n}^{\fr}}
\newcommand{\Tc}{T_\text{col}}

\DeclareMathOperator{\sign}{sign}

\newcommand\double[1]{
 \raisebox{-0.7em}{\begin{overpic}[page=#1]{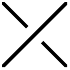}
 \put (21,38) {\scriptsize$i$}
 \put (72,38) {\scriptsize$j$}
\end{overpic}}}

\newcommand\reid[1]{\raisebox{-1.0em}{\includegraphics[page=#1]{images.pdf} }}

\newcommand\icon[1]{\ensuremath{\raisebox{-0.7em}{\includegraphics[page=#1]{images.pdf}}}} % kauff. relation
\newcommand\iconk[1]{\ensuremath{\Big[\icon{#1}\Big]}} % kauff. relation
\newcommand\bigicon[1]{\ensuremath{\raisebox{-1.1em}{\includegraphics[page=#1]{images.pdf}}}}
\newcommand\exicon[1]{\ensuremath{\raisebox{-1.1em}{\includegraphics[page=#1]{images.pdf}}}} 

\newcommand\extrafootertext[1]{%
    \bgroup
    \renewcommand\thefootnote{\fnsymbol{footnote}}%
    \renewcommand\thempfootnote{\fnsymbol{mpfootnote}}%
    \footnotetext[0]{#1}%
    \egroup
}

\begin{document}

\title{Invariants of multi-linkoids}
\author[1]{Bo\v{s}tjan Gabrov\v{s}ek\thanks{bostjan.gabrovsek@fs.uni-lj.si}
}
\affil[1]{Faculty of Mechanical Engineering and Faculty of Education, University of Ljubljana, Ljubljana, Slovenia}
%\email[1]{bostjan.gabrovsek@fs.uni-lj.si}

\author[2]{Neslihan G\"ug\"umc\"u\thanks{neslihangugumcu@iyte.edu.tr}}
\affil[2]{Department of Mathematics, Izmir Institute of Technology, Izmir, Turkey}
%\email[2]{​neslihangugumcu@iyte.edu.tr}
%\thanks{none}

%\subjclass[2020]{Primary 57K12, secondary 57M15}

%\keywords{knotoid, multi-linkoid, spatial graph, Kauffman bracket polynomial, Kauffman bracket skein module, theta-curve, theta-graph}

%{
%  \small	
%  \textbf{\textit{Keywords --}} knotoid, multi-linkoid, spatial graph, Kauffman bracket polynomial, Kauffman bracket skein %module, theta-curve, theta-graph
%}

\date{\today}

\maketitle

\begin{abstract}
In this paper, we extend the definition of a knotoid that was introduced by Turaev, to multi-linkoids that consist of a number of knot and knotoid components. We study invariants of multi-linkoids that lie in a closed orientable surface, namely the Kauffman bracket polynomial, ordered bracket polynomial, the Kauffman skein module, and the $T$-invariant in relation with generalized $\Theta$-graphs.
\end{abstract}

\extrafootertext{{\it 2020 MSC.} Primary 57K12, secondary 57M15.}
\extrafootertext{{\it Keywords.} Knotoid, multi-linkoid, spatial graph, Kauffman bracket polynomial, Kauffman bracket
skein module, theta-curve, theta-graph.}

\section*{Introduction}

Knotoids were defined by Turaev~\cite{Turaev2012} as immersions of the unit interval in an arbitrary surface. They can be regarded as open-ended knot diagrams with two endpoints, on which we consider an equivalence relation induced by the Reidemeister moves.
In this sense, the theory of knotoids in $S^2$ generalizes classical knot theory and has a natural connection with the theory of virtual knots~\cite{Kauffman1999, Kauffman2012} through the \emph{virtual closure}. The intrinsic invariants of knotoids, as well as how knot invariants extend as knotoid invariants, have been studied by many researchers.
%See \cite{Turaev2012, Gueguemcue2017a, Gueguemcue2021a, Gueguemcue2021b, Barbensi2021, Kutluay2020, Barbensi2018, Tarkaev2021, Gueguemcue2019, Gueguemcue2021, Korablev2017, Bartolomew2021} for studies on knotoids. 
See \cite{Barbensi2018,  Barbensi2021, Bartolomew2021, Diamantis2022, Gueguemcue2017a, Gueguemcue2021b,  Gueguemcue2021a, Gueguemcue2019, Gueguemcue2021, Korablev2017, Kutluay2020, Tarkaev2021, Turaev2012} for studies on knotoids. 

Knotoids can be considered in a more topological setting, 
as planar knotoids are equivalent to relative knots in $\mathbb{R}^3 \setminus \text{\{two parallel lines\}}$~\cite{ Gueguemcue2017a, Przytycki1991}. In relation to this topological setting, Turaev showed that knotoids in $S^2$ are in one-to-one correspondence with simple $\Theta$-curves and multi-knotoids in $S^2$, immersions of the unit interval and a finite number of circles in $S^2$, are in one-to-one correspondence with simple theta-links~\cite{Turaev2012}.

Topological structures are also an emerging field in modern chemistry~\cite{Forgan2011} as knots have been identified DNA~\cite{Siebert2017} and proteins~\cite{Liang1994, Virnau2006}. Research suggest that knots increase thermal and kinetic stability of the molecule~\cite{Sulkowska2008}, as well as have important functional roles~\cite{Virnau2006}. A protein's backbone natively forms an embedded interval in 3-space. Classical studies of protein topology (i.e.\ determining the knot type) rely on closing this interval by some closure method (closure methods are thoroughly discussed in~\cite{Millett2013}).
Since closures are often ambiguous (in the case when the protein termini are not located close to the minimal convex surface enveloping the protein), knotoids have been identified as natural candidates to study open-knotted proteins~\cite{Goundaroulis2017, Goundaroulis2017a, Gueguemcue2017}. 

One can extend the notion of a knotoid by considering several closed and open-ended components. 
A \emph{multi-linkoid} is a union of a finite number of immersed unit intervals and circles in a closed orientable surface. 

We expect that multi-linkoids would suggest a new setting for the topological analysis of several mutually entangled polymer chains or subchains via the invariants we introduce here.

The main goal of this paper is to generalize the mentioned concepts to multi-linkoids and to introduce invariants for them.

The paper is organized as follows. 
% overview
Section~\ref{sec:pre} is an overview of the required notions related to multi-linkoids. 
% bracket
In Section~\ref{sec:kauffman} we extend the Kauffman bracket polynomial to multi-linkoids. In addition, we strengthen this invariant to the ordered Kauffman bracket polynomial for multi-linkoids with an ordering on its knotoid components. 
% kbsm
In Section~\ref{sec:kauffman2} we introduce the Kauffman skein module of multi-linkoids and show that the module if freely generated. 
% theta
In Section~\ref{sec:spatial} we study multi-linkoids in the topological setting. We show that multi-linkoids are equivalent to relative links in $\mathbb{R}^3$ and introduce simple generalized $\Theta$-graphs, which are in one-to-one correspondence with multi-linkoids, where we take also into consideration the possibility of an ordering on the components.

In Section~\ref{sec:T} we utilize the $T$-invariant for spatial graphs~\cite{Kauffman1989} and strengthen it by introducing the colored $T$-invariant for distinguishing multi-linkoids and ordered multi-linkoids.

\section{Preliminaries}\label{sec:pre}
\begin{definition}\normalfont

A \emph{multi-linkoid diagram} in a closed (oriented or unoriented) surface $\Sigma$ is an immersion of a number of unit intervals $[0,1]$ and unit circles $S^1$ into $\Sigma$. This immersion is generic in the sense that there are only a finite number of intersections of the image that we endow each with under or over information, and regard them as \emph{crossings} of the multi-linkoid diagram.

The images of the points $0$ and $1$ of the unit intervals are considered to be distinct from each other and are called the endpoints of the diagram. We consider an orientation on each of the components of a multi-linkoid diagram in a way that each open-ended component is oriented from the image of $0$ named specifically as \emph{tail} to the image of $1$ named specifically as \emph{head}. The special cases of multi-linkoids are 
\emph{knotoids} consisting of only one open (knotoid) component and \emph
{multi-knotoids} consisting of only one open component and a number of closed (knot) components.

\end{definition}
\begin{figure}[H]
\centering

\includegraphics[scale=0.2]{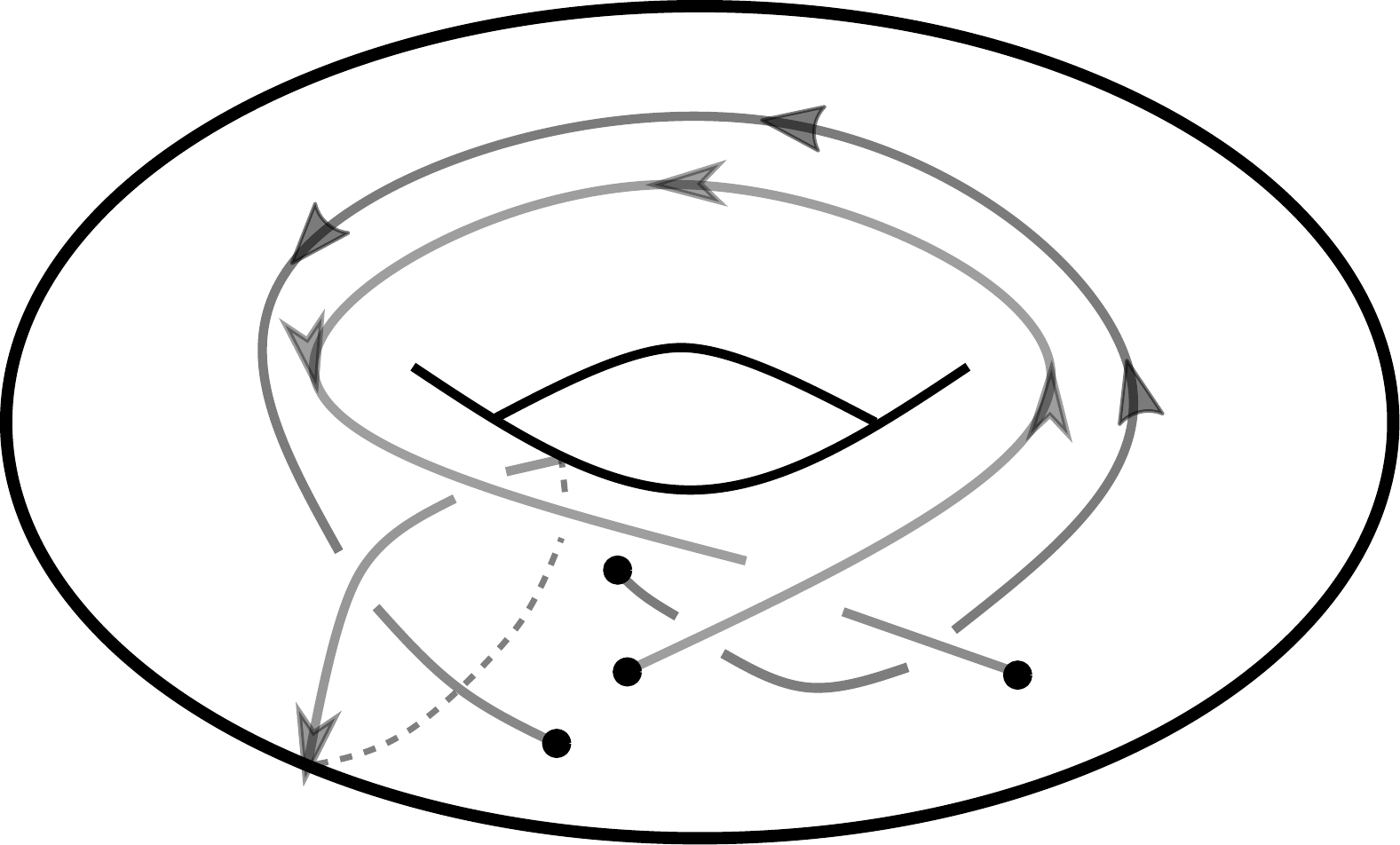}
\caption{A multi-linkoid diagram with three components on the torus.}
\label{}

\end{figure}

As in the case of knotoids~\cite{Turaev2012}, we consider multi-linkoids up to the equivalence relation generated by the Reidemeister moves.
\begin{definition}
Two multi-linkoids are \emph{equivalent }if and only if they differ by a finite sequence of Reidemeister moves R-0 (surface isotopy),  R-I, R-II, and R-III, shown in the Figure~\ref{fig:reid}.  It is not allowed to move the endpoints over or under a strand as depicted in Figure~\ref{fig:forb}. 
\end{definition}

\begin{figure}[ht]
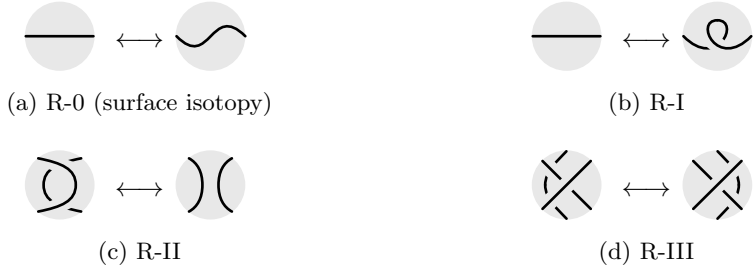

     \centering
     \begin{subfigure}[b]{0.43\textwidth}
         \centering
         \reid{41} $\longleftrightarrow$ \reid{42}
         \caption{R-0 (surface isotopy)}
         \label{fig:r0}
     \end{subfigure}
     \begin{subfigure}[b]{0.43\textwidth}
         \centering
         \reid{41} $\longleftrightarrow$ \reid{43}
         \caption{R-I}
         \label{fig:r1}
     \end{subfigure}\\[1em]
     \begin{subfigure}[b]{0.43\textwidth}
         \centering
         \reid{24}  $\longleftrightarrow$  \reid{31}
         \caption{R-II}
         \label{fig:r2}
     \end{subfigure}
     \begin{subfigure}[b]{0.43\textwidth}
         \centering
         \reid{32}  $\longleftrightarrow$  \reid{35}
         \caption{R-III}
         \label{fig:r3}
     \end{subfigure}

        \caption{Local Reidemeister moves for multi-linkoids on a surface $\Sigma$.}
        \label{fig:reid}
\end{figure}

\begin{figure}[ht]
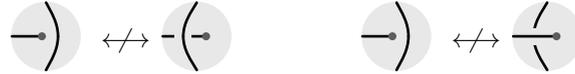

     \centering
         \reid{38} $\centernot\longleftrightarrow$ \reid{39}  \qquad\qquad \reid{38} $\centernot\longleftrightarrow$  \reid{40}
         \caption{The forbidden move.}
        \label{fig:forb}
\end{figure}

In~\cite{Moltmaker2021}  \emph{framed knotoids} and their quantum invariants were introduced. 
\begin{definition}
A \emph{framed knot} is a knot equipped with a transversal, smooth, everywhere nonzero vector field (or equivalently, an embedding of a ribbon/annulus in the 3-space).
\end{definition}
%The framing is defined to be the associated element in $\pi_1(SO(2)) \cong \mathbb{Z}.$
%Isotopy of framed knots is generated by moves R-0, FR-I, R-II, and R-III (presented in Figures~~\ref{fig:reid} and~\ref{fig:framed1}).
The move R-I changes the framing of a knotoid, we thus replace it with the move FR-I depicted in Figure~\ref{fig:framed1}.
Framed multi-linkoids are now naturally defined by the following definition.

\begin{definition}\normalfont
A \emph{framed multi-linkoid} on an orientable surface $\Sigma$ is an equivalence class of knotoid diagrams under
the equivalence generated by R-0, FR-I, R-II, and R-III.% (see Figures~~\ref{fig:reid} and~\ref{fig:framed1}).
\end{definition}

In an oriented multi-linkoid diagram we assign to each crossing a sign using the convention $\sign\Big(\icon{62}\Big) = 1$ and $\sign\Big(\icon{63}\Big) = -1$.

\begin{definition}\normalfont
The sum of signs over all crossings of an oriented multi-linkoid diagram of $L$ is called the \emph{writhe}, $w(L)$, of $L$. 
\end{definition}

Note that moves FR-I, R-II, and R-III do not change the writhe, $w(K)$ is thus an invariant of framed oriented multi-linkoids.

\begin{figure}[ht]
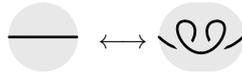

     \centering
  \reid{41} $\longleftrightarrow$ \reid{61}
         \caption{The move FR\textnormal{-}I (equivalent to the move $R1'$ in \cite{Moltmaker2021}.)}
        \label{fig:framed1}
\end{figure}

\begin{definition}\normalfont
An \emph{ordered multi-linkoid} diagram is a multi-linkoid diagram equipped with an ordering on its open components. 
%An ordering can also be considered only for the open (knotoid) components  of a multi-linkoid diagram. We call such multi-linkoid diagram a \emph{(pre?)- ordered multi-linkoid diagram}.
%{\color{red} Perhaps partially ordered sound a bit limiting, also we really only consider partially ordered links; perhaps vertex-ordered or endpoint-ordered?}
We consider ordered multi-linkoids up to the equivalence relation generated by the Reidemeister moves.

\end{definition}
The ordering on the $n$ open components of a multi-linkoid diagram induces an ordering of the endpoints. Precisely, each endpoint is enumerated by an integer $i$ from $\{1, 2, \ldots, 2n\}$, starting from the first component with respect to the given ordering. In the sequel, some invariants of multi-linkoid diagrams are constructed with respect to the ordering induced at the endpoints, see Section \ref{sec:kauffman} for the ordered Kauffman bracket, Section \ref{sec:kauffman2} for the ordered Kauffman bracket skein module and Section \ref{sec:T} for the invariant $\Tc$.

%\subsection{Closures of a multi-linkoid diagram}

%Q. What is the exact correspondence between linkoids and virtual links?\\
%Given a multi-linkoid diagram in $\mathbb{R}^2$ or $S^2$. One can connect the endpoints of each open-ended component   
%together with an embedded arc. If the closure arcs meet with strands of the diagram, we declare each such intersection as virtual crossings. Such closure of the multi-linkoid diagram is its \emph{virtual closure}, and in fact, it is an oriented virtual link diagram.
%\begin{proposition}
%The virtual closure induces a surjection between the set of multi-linkoids in $S^2$ and virtual links.

%\end{proposition}
%\begin{proof}

%\end{proof}

%\begin{proposition}
%The virtual closure induces a bijection between the set of  multi-linkoids in $S^2$ and virtual links.

%\end{proposition}

\section{Kauffman bracket polynomial of multi-linkoids}\label{sec:kauffman}

In this section we extend the Kauffman bracket polynomial of knotoids to multi-linkoids in $S^2$ or $\mathbb{R}^2$.

Let $L$ be a multi-linkoid diagram in $S^2$ with $n$ crossings. Without any consideration of orientation on the components of $L$, we smooth all the crossings of $L$ by $A$-type and $B$-type smoothings that are determined by the four local regions adjacent to the crossings. The two regions that are swept by $90$ degrees rotation of the overpassing strand of a crossing in the counterclockwise direction are labeled with the letter $A$, and the remaining two regions are labeled with the letter $B$.  An $A$-\emph{type smoothing} of a crossing is to remove the crossing and connect two of the $A$-regions and a $B$-\emph{type smoothing} of a crossing is to remove the crossing and connect two of the $B$-regions that adjacent to the crossing. See Figure \ref{fig:smoothing}. 

The resulting collection of curves from a chosen smoothing type of each crossing of $L$ is called a \emph{state} of $L$ and denoted by $\sigma_i$ where $i \in \{0,1\}^n$.  Each state of $L$ contains a number of simple closed curves and simple arcs in $S^2$ containing the endpoints of $L$. Each circle component is then assigned the value $-A^2 - A^{-2}$ and each open component is assigned the variable $\lambda$ cofactored by the product of the labels $A$'s and $B$'s coming from the smoothing types made to obtain $\sigma$. We take $B = A^{-1}$ as in the knotoid bracket case~\cite{Turaev2012},  and we obtain the following Laurent polynomial in the variables $A, A^{-1}$ for $L$.
%Each circle component of $\sigma$ is assigned to a value determined the labels  In this way, we obtain the following Laurent polynomial

\begin{figure}[H]
\centering
\includegraphics[scale=0.17]{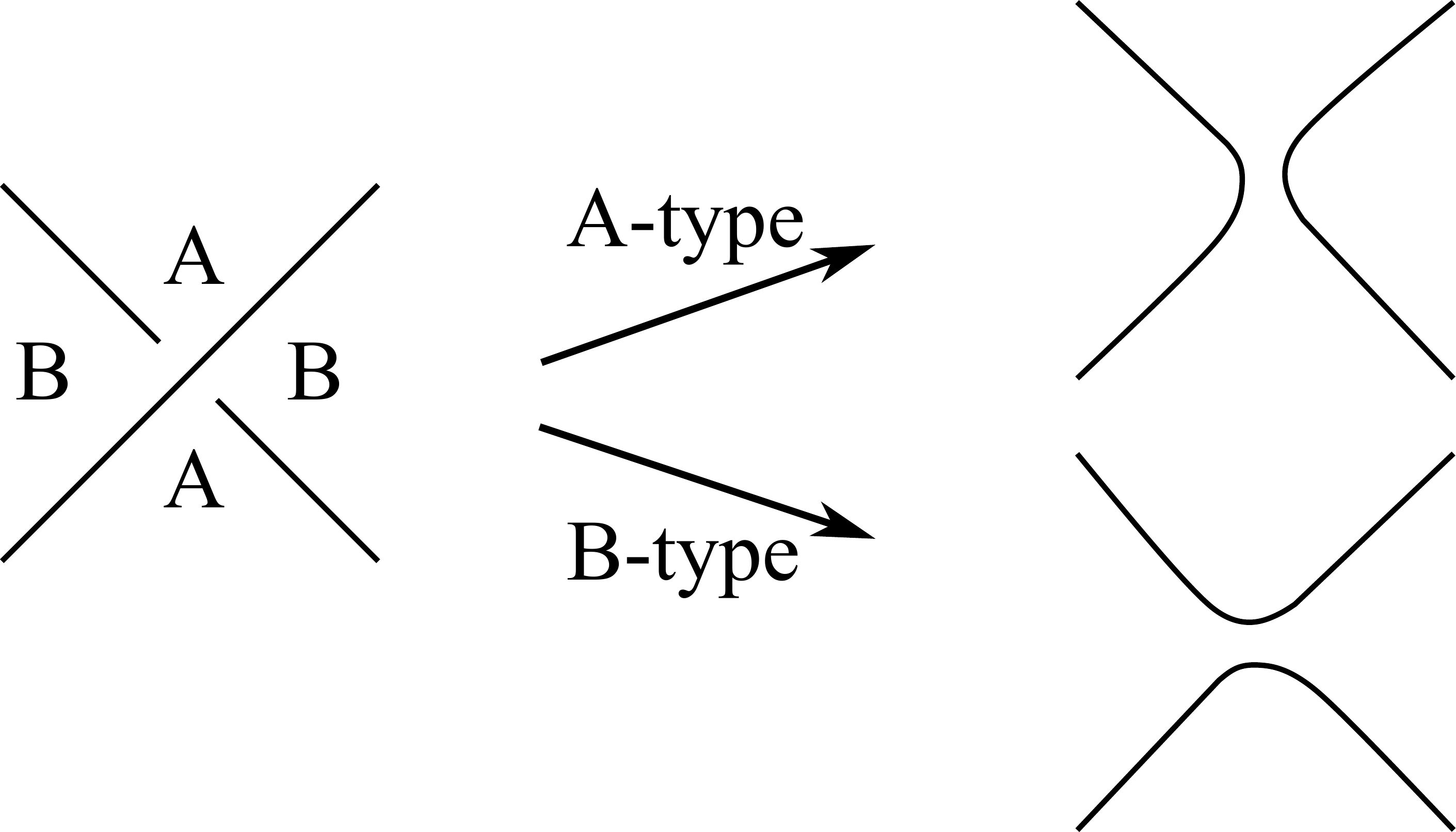}
\caption{Smoothing types of a crossing}
\label{fig:smoothing}

\end{figure}

\begin{definition}\normalfont

The \emph{Kauffman bracket polynomial} of $L$ is defined to be the sum
\begin{equation}
<L> (A^{\pm 1}, \lambda) = \Sigma _{\sigma_i} < L \mid  \sigma_i > (-A^2 - A^{-2} )^{ ||c|| } \lambda^{||l||},
\label{eq:kb}
\end{equation}
where $<L \mid \sigma >$  is the product of the labels of the state $\sigma_i$, $||c||$ is the number of closed components of $\sigma_i$ and $||l||$ is the number of open components of $\sigma_i$. 
\end{definition}

In the case of a multi-knotoid diagram, we obtain the usual Kauffman bracket polynomial by taking $\lambda=1$.

\begin{proposition}
The bracket polynomial is an invariant of framed multi-linkoids.
% It is a multi-linkoid invariant when it is normalized (multiplied) by the factor $(-A^3)^{-w(L)}$, where $w(L)$ is the writhe.
%In other words, 
The normalization $$\overline{<L>} (A, \lambda) = (-A^3)^{-w(L)} <L> (A, \lambda),$$ is an invariant of multi-linkoids.
\end{proposition}
\begin{proof}
Since the Reidemeister moves take place locally away from endpoints, the proof of the invariance runs similarly with the case of invariance of the bracket polynomial of knotoids under R-II and R-III moves, as given in  \cite{Turaev2012}.  Here we illustrate how the bracket polynomial changes under a R-I move for a knotoid diagram $D$. It is clear that the writhe of the given diagram, when considered to be oriented, is increased by one under the R-I move. 
Therefore, the multiplicative factor $-A^3$ added by the R-I move in the bracket polynomial can cancelled by the factor $(-A^3)^{-w(D)+1}$, where $-w(D)+1$ is clearly the writhe of the resulting diagram by the R-I move. See Figure~\ref{fig:beh}.
\end{proof}

\begin{figure}[H]
\centering
\includegraphics[scale=0.26]{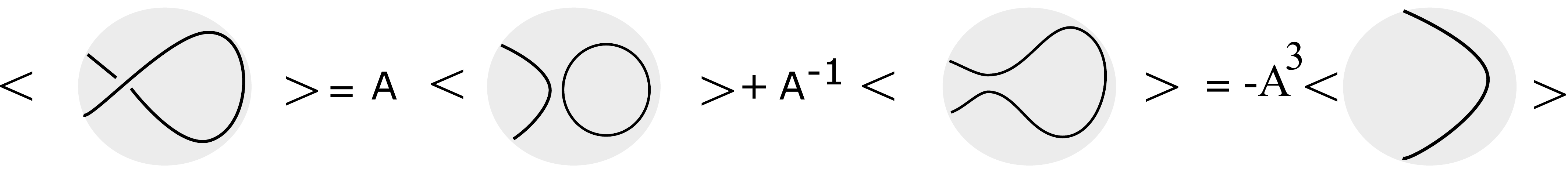}
\caption{Behaviour of the Kauffman bracket polynomial under a R-I move.}
\label{fig:beh}
\end{figure}

%We define the normalized Kauffman bracket of a multi-linkoid diagram $L$ as
%$$\overline{<L>}(A, \lambda) = (-A^3)^{-w(L)}\cdot <L>(A, \lambda)$$

%\emph{\bf The ordered bracket polynomial}\label{sec:OBP}

Let us consider an ordered version of the Kauffman bracket polynomial. 
Let $L$ be an ordered multi-linkoid diagram in $S^2$. We enumerate each of its endpoints by an integer  $i \in \{1,\ldots, 2n\}$, $n \geq 1$, according to the ordering on the open components of $L$.
  We apply the bracket smoothing at each crossing of $L$ to obtain the states of $L$.  Each open component of a state is assigned $\lambda_{ij}$, $i < j$ where $i,j$ are the integer labels at the endpoints of the components. 
 Circle components of a state are again assigned the value $-A^2-A^{-2}$. We obtain the following Laurent polynomial in variables $A, A^{-1}, \lambda_{ij}$ ($i < j$). 
  
  %{\color{red} Should it be open segment, not long segment, since we did not define a long segment?}

\begin{definition}\normalfont
The \emph{ordered bracket polynomial} of $L$, $<L>_{\bullet}$ is defined to be the sum
\begin{equation}
<L>_{\bullet} (A^{\pm 1}, \{\lambda_{ij}\}_{i < j}) = \Sigma _{\sigma} <K \mid \sigma > (-A^2 - A^{-2} )^{ ||c|| } \prod_{\Lambda} \lambda_{ij},
\label{eq:okb}
\end{equation}
where the sum is taken over all states of $L$, $<K \mid \sigma >$ is the product of the smoothing labels of a state $\sigma$ and $\Lambda$ is the collection of open components in the state $\sigma$.
\end{definition}

Note that if $L$ is a multi-knotoid diagram,  the ordering on $L$ is trivial since there is only one knotoid component of $L$. Then, the ordered bracket polynomial of $L$ can be assumed to be equal to the Kauffman bracket polynomial of $L$ with $\lambda_{12}= \lambda=1$.

\begin{proposition}
The ordered bracket polynomial is an framed ordered multi-linkoid invariant. 
The normalization of the ordered bracket polynomial, $$\overline{<L>}_{\bullet} (A^{\pm 1}, \{\lambda_{ij}\}_{i< j}) = (-A^3)^{-w(L)} <L> (A, \{\lambda_{ij}\}_{i,j}),$$ is an invariant of ordered multi-linkoids.
%It turns to be a multi-linkoid invariant when it is multiplied with the factor $(-A^3)^{-w(L)}$.
\end{proposition}
\begin{proof}
It is clear that the Reidemeister moves FR-I, R-II, and R-III preserve both the framing and ordering on a multi-linkoid diagram, and the invariance of the polynomial under these moves follows similarly as the invariance of the bracket polynomial. The ordered bracket polynomial behaves the same under an R-I move as the bracket polynomial, which implies that the normalization by the factor $(-A^3)^{-w(L)}$ becomes an ordered multi-linkoid invariant. 

\end{proof}

%\begin{example}\normalfont
%Figure \ref{fig:ex1} illustrates the computation of the ordered bracket polynomial of an ordered linkoid diagram with two components.

%\begin{figure}[H]
%\centering

%\includegraphics[scale=0.2]{orderedex.pdf}
%\caption{An example of computation of the ordered bracket polynomial}
%\label{fig:ex1}

%\end{figure}

%\end{example}

\begin{example}\normalfont
In Figure \ref{fig:ex2} two ordered linkoid diagrams with two components are given. Explicit computation shows that the normalized Kauffman bracket polynomial of these linkoids coincide, but they can be distinguished by the normalized ordered bracket polynomial. Precisely, we find,
$<L_1>_{\bullet} = (A^2 +1) \lambda_{12} \lambda_{34} + \lambda_{13}\lambda_{24}  + A^{-2} \lambda_{14}\lambda_{23}$ and $<L_2>_{\bullet} = (A^2 + 1) \lambda_{12}\lambda_{34} + \lambda_{14}\lambda_{23} + A^{-2} \lambda_{13}\lambda_{24}$.

\begin{figure}[H]
\centering

\includegraphics[scale=0.2]{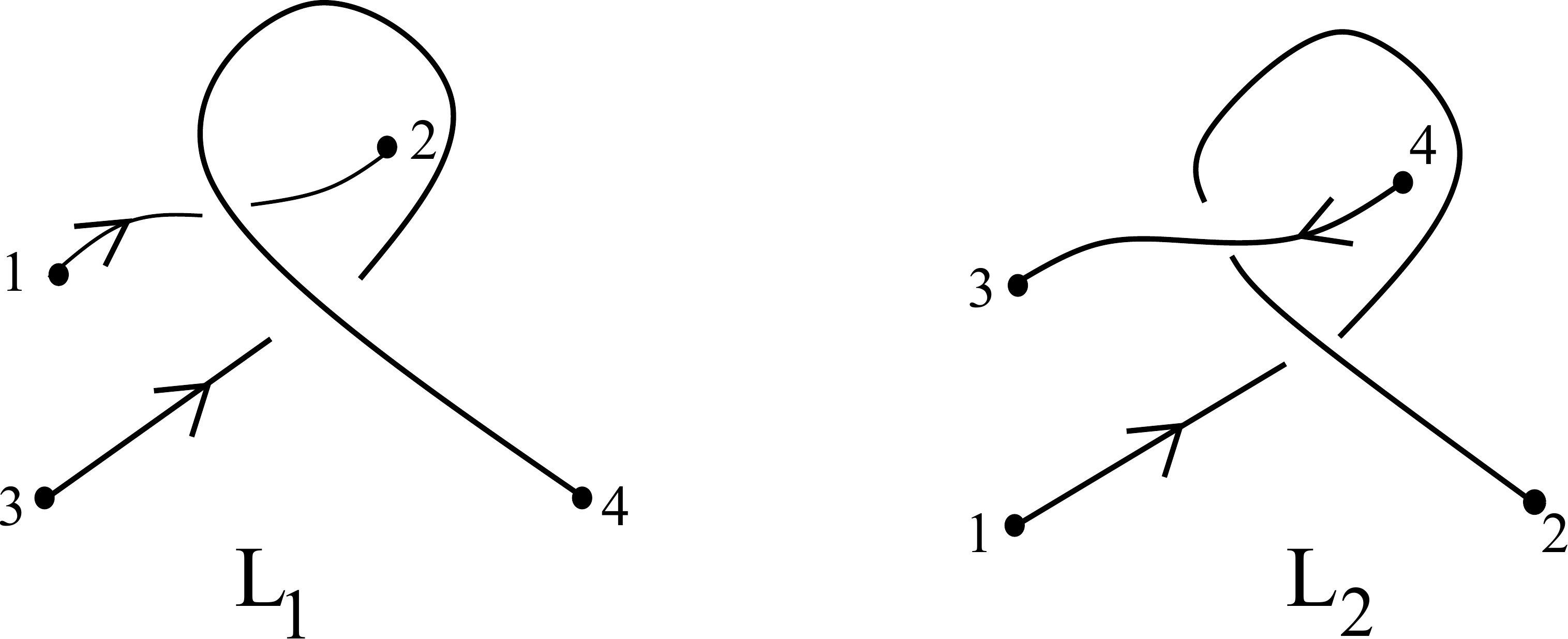}
\caption{Two ordered linkoids that are distinguished by the ordered bracket polynomial }
\label{fig:ex2}

\end{figure}

\end{example}

%Let $L_1$ and $L_2$ be two linkoid diagrams both with two components and same type of crossings, and suppose that we give an ordering on $L_1$ and $L_2$ oppositely. It is an easy observation that the (unordered) bracket states of $L_1$ and $L_2$ coincide with the same multiplicative terms. The exchange of ordering on the components of $L_1$ swaps the integer labels on the endpoints of its state components but does not change the ordered bracket polynomial. In this way we observe that the ordered bracket polynomials of $L_1$ and $L_2$ coincide.

\subsection{The Kauffman bracket skein module}\label{sec:kauffman2}

Observe that in the bracket polynomial formulas~\eqref{eq:kb} and~\eqref{eq:okb}, 
we only consider the number of closed components $||c||$ in each state, where each component is assigned the term $-A^2-A^{-2}$.
We can extend the bracket polynomial so that in each state we also keep information about the homology classes of elements in $c$ in the complement $S^2 \setminus l$.
We do this by introducing the Kauffman bracket skein module (KBSM) of multi-linkoids. Furthermore, we extend this invariant to multi-linkoids in any closed, connected, orientable genus $g$ surface.

Skein modules were independently introduced by Turaev~\cite{Turaev1990} and Przytycki~\cite{Przytycki1991}, they can be viewed as generalizations of invariants based on the skein relation for knots in 3-manifolds.
%In this section, we will adopt the notion of a skein module and introduce the Kauffman bracket skein module (KBSM) for multi-linkoids in an arbitrary closed, connected, orientable surface $g$.
The idea behind our construction (which is closely related to the original construction) is that we first construct a space of all possible linear combinations of multi-linkoids and in this space impose the skein and framing relation, which characterize the bracket polynomials given by formulas~\eqref{eq:kb} and \eqref{eq:okb}.
For similar constructions see~\cite{Diamantis2019, Diamantis2016a, Gabrovsek2017, Gabrovsek2021, Mroczkowski2018}.

\begin{definition}\normalfont
Let $\Sigma$ be a closed connected orientable surface of genus $g$.
Let $R$ be a commutative ring with an invertible element $A$ (e.g.\ the ring of Laurent polynomials $\mathbb{Z}[A,A^{-1}]$) and
let $\Lf$ be the set of framed multi-linkoids on $\Sigma$ with $2n$ endpoints.
Denote by $R[\Lf]$ the free $R$-module spanned by $\Lf$ and by $\submodule(\Lf,R,A)$ the submodule of $R[\Lf]$ generated by the following two expressions (relators):
\begin{subequations}
\label{eq:kauff}
\begin{align}
         \iconk{6} - A \iconk{7} - A^{-1} \iconk{8},       \label{eq:kauff1} \\
       \iconk{9} - (A^2-A^{-2})\iconk{10}, \label{eq:kauff2}
\end{align}
\end{subequations}
where  \iconk{6},  \iconk{7}, and \iconk{8} (resp. \iconk{9} and \iconk{10}) represent classes of multi-linkoids that are everywhere the same except inside a small disk where they look like the figures indicated.

The \emph{Kauffman bracket skein module of multi-linkoids in $\Sigma$ with $2n$ endpoints} is the quotient module 
$$\Kf(R,A) = R[\Lf] / \submodule([\Lf;R,A),$$
i.e.\ all formal finite linear sums of multi-linkoids in which we enforce the two relations obtained by the expressions~\eqref{eq:kauff}. 
\end{definition}

%Similarly, if we take $\rsubmodule(R,l,m)$ to be the submodule of  $R[\rlinks]$ generated by the HOMFLYPT skein expression~\eqref{eq:skein}, 
%the \emph{HOMFLYPT skein module of rigid colored bonded links} is the module 
%$$\rhsm(R,l,m) = R[\rlinks] / \rsubmodule(R,l,m).$$

\begin{theorem}\label{thm:unmain}
Let $\Bf$ be the set of all multi-linkoids in $\Sigma$ with $2n$ endpoints without crossings and without trivial contractible components.
The module $\Kf(R,A)$ is freely generated by $\Bf$, i.e.\ $\Kf(R,A) = R[\Bf]$.
\end{theorem}

In particular
\begin{itemize}
\item $\kbsm_{S^2,1}$ is freely generated by the trivial knotoid
$\raisebox{-0.5em}{\includegraphics[page=11]{images.pdf}}$,
\item $\kbsm_{\R^2,1}$ is generated by the infinite set 
$\big\{ \raisebox{-0.8em}{\includegraphics[page=12]{images.pdf}}, \raisebox{-0.7em}{\includegraphics[page=13]{images.pdf}}, \raisebox{-0.7em}{\includegraphics[page=14]{images.pdf}}, \ldots \big\}$,
\item  $\kbsm_{\R^2, 2}$ is generated by the infinite set
$\raisebox{-1.5em}{\includegraphics[scale=.2]{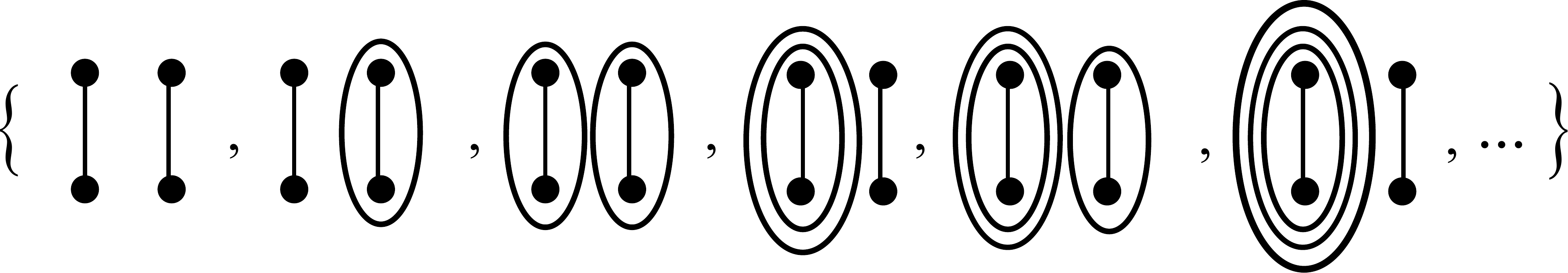}}$.
\item $\kbsm_{T^2, 1}$ is generated by the infinite set
$\raisebox{-1.5em}{\includegraphics[scale=.25]{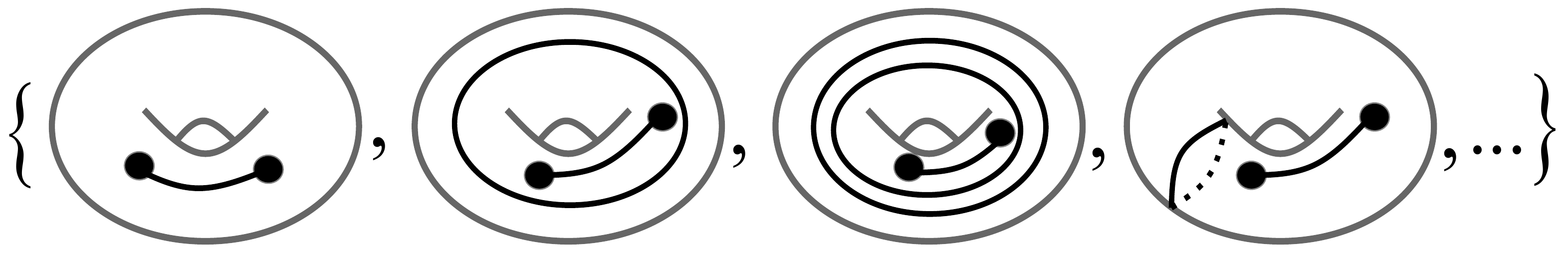}}$.
\end{itemize}

\begin{proof}
In order to prove the theorem, it is enough to show that $\Bf$ is a basis for $\Kf(R,A)$ and that for a given multi-linkoid $L \in \Lf$ the expression $[L] \in \Kf(R,A)$, written in terms of the basis, is unique.

For a given representative $L$ of $[L] \in \Kf(R,A)$, we can first remove all crossings using~\eqref{eq:kauff1} and then remove all trivial components using~\eqref{eq:kauff2}, we end up with a formal linear sum of elements $\Bf$. Elements $\Bf$ thus generate of the module.

To show that the expression $[L]$, evaluated in $\Bf$, is unique, 
we need to show that it does not depend on any choice we can make during the computation of $[L]$ and that it is invariant under Reidemeister moves.
We enumerate the crossings of $L$ by ordinals $1, 2, \ldots, k$. We first show that the result does not depend on the order we perform crossing eliminations via~\eqref{eq:kauff1}.

Let \double{15} represent a multi-linkoid with crossings $i$ and $j$ marked. The result does not depend on the order of crossings smoothings:
$$\double{15} = A\double{16} + A^{-1}\double{17} =  A^2\double{18} + \double{19} + \double{20} + A^{-2}\double{21},$$
$$\double{15} = A\double{22} + A^{-1}\double{23} =  A^2\double{18} + \double{19} + \double{20} + A^{-2}\double{21}.$$
In addition, the result does not depend on the elimination order of trivial components using the framing relation~\eqref{eq:kauff2}.

Proving invariance is similar to that in the classical case. The expression is invariant under Reidemeister move R-II:
$$\reid{24} = A \reid{26}+ A^{-1}\reid{25} = A^2\reid{30} + \reid{29} + \reid{27}+A^{-2}\reid{28} = \reid{31},$$
where we used the framing relation~\eqref{eq:kauff2} in the last equality.
The expression is also invariant under Reidemeister move R-III:
$$\reid{32} = A \reid{33}+ A^{-1}\reid{34} = A\reid{36} + A^{-1}\reid{34} = \reid{35},$$
where the second equality holds by invariance under R-II.
\end{proof}

Let $L$ be a multi-linkoid in $\Sigma$ with $2n$ endpoints. We denote by $[L]_{\Bf}$ the class of $L$ in $\Kf$ written in terms of elements in the basis $\Bf$ (or just $[L]$ if we fix the basis and the ambient space is known from context).
Due to Theorem~\ref{thm:unmain}, $[L]$ is an invariant of ordered framed multi-linkoids.

As in the classical case, we can obtain an invariant of non-framed linkoids by multiplying it by $(-A^3)^{-w(L)}$.
The expression
$$\overline{[L]}_{\Bf} = (-A^3)^{-w(L)} \, [L]_{\Bf}$$
is an invariant of multi-linkoids in $\Sigma$.

Let us now consider the ordered case.
For our construction, we will need the set of multi-linkoids with arbitrary ordering on the vertices (not necessarily consecutive on the endpoints on the same component), i.e.\ each endpoint is assigned exactly one value in $\{1,2,\ldots,2n\}$. We call such multi-linkoids \emph{vertex-ordered multi-linkoids.}

Let the set $\hatLf$ of all \emph{vertex-ordered multi-linkoids} with $n$ open components and repeat the construction.
The quotient module 
$$\hatKf(R,A) = R[\hatLf] / \submodule([\hatLf;R,A)$$
is the \emph{Kauffman bracket skein module of vertex-ordered multi-linkoids in $\Sigma$ with $2n$ endpoints}.
As ordered multi-linkoids are just special cases of vertex-ordered multi-linkoids, we will obtain an ordered multi-linkoids invariant through $\hatKf(R,A)$.

%Let us now consider the case vertex-ordered multi-linkoids.
%Let $\hatBf$ be the set of all partially ordered multi-linkoids in $\Sigma$ with $2n$ endpoints without crossings and without trivial contractible components.

\begin{theorem}\label{thm:main}
Let $\hatBf$ the set of all isotopy classes of vertex-ordered multi-linkoids in $\Sigma$ with $2n$ endpoints without crossings and without trivial contractible components.
The module $\hatKf(R,A)$ is freely generated by the basis $\hatBf$, i.e.\ $\hatKf(R,A) = R[\hatBf]$.
\end{theorem}

\begin{proof}
The proof is essentially the same to that of Theorem~\ref{thm:unmain}. 
We can reduce every link $L \in \hatLf$ to a formal sum of elements from $\hatBf$ by operations~\eqref{eq:kauff}. Clearly, the result also does not depend on the order we perform the operations.
To show invariance under Reidemeister moves, we repeat the argument from the proof of Theorem~\ref{thm:unmain} verbatim.
%However, we will prove it using the fact that there is a natural action of the symmetric group on $\hatKf(R,A)$.
%The symmetric group $S_{2n}$ acts on the ordering of the endpoints of elements in $\hatLf$. This action also permutes elements in the basis $\hatBf$ and the module 
% $\Kf(R,A) / S_{2n}$ is freely generated by the orbits of $S_{2n}$, which coincide with $\Bf$. It follows that  $\Kf(R,A) \cong \hatKf(R,A) / S_{2n}$, that is,  
%  $\Kf(R,A)$ is freely generated by $\Bf$.
%  {\color{red} proof not ok, since 12 = 21 in the case n=1}
\end{proof}

As before, the normalized expression
$$\overline{[[L]]}_{\hatBf} = (-A^3)^{-w(L)} \, [L]_{\hatBf}$$
is an invariant of vertex-ordered multi-linkoids on $\Sigma$ and, as a special case, an invariant of ordered multi-linkoids on $\Sigma$.

The following propositions follow directly from construction.

\begin{proposition}\label{prop:kbsm-vs-kbp-un}
Given an (unordered) multi-linkoid $L$ in $S^3$ (or $\R^3$), the Kauffman bracket polynomial $<L> (A, \lambda)$
is obtained from $[L]$ by
replacing each open component with $\lambda$ and every closed component by the factor $(-A^2-A^{-2})$.
\end{proposition}

\begin{proposition}\label{prop:kbsm-vs-kbp-or}
Given an ordered multi-linkoid $L$ in $S^3$ (or $\R^3$), the ordered Kauffman bracket polynomial $<L>_\bullet (A, \{\lambda_{ij}\}_{i,j})$
is obtained from $[[L]]$ by
replacing each open component with endpoints, enumerated by $i$ and $j$, by $\lambda_{i,j}$ and replacing each closed component with the factor $(-A^2-A^{-2})$.
\end{proposition}

\begin{example}\normalfont

\begin{figure}[ht]
     \centering
     \begin{subfigure}[b]{0.30\textwidth}
         \centering
        \includegraphics[page=83]{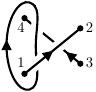}
         \caption{$K_1$}
         \label{fig:mk1}
     \end{subfigure}
     \begin{subfigure}[b]{0.30\textwidth}
         \centering
         \includegraphics[page=84]{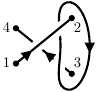}
         \caption{$K_2$}
         \label{fig:mk2}
     \end{subfigure}
        \caption{Two oriented multi-knotoids on $\R^2$.}
        \label{fig:mk}
\end{figure}

Consider the two oriented multi-linkoids $K_1$ and $K_2$ in $\R^2$ in Figure~\ref{fig:mk}.
We have $w(K_1) = w(K_2)=-1$. A straightforward computation shows us:
%\begin{equation*} \label{kbsm1}
%    \begin{split}
%    [[K_1]]_{\hat\basis_{\R^2,1}} & = (-A^3+A^7) \exicon{66} + A^{-3}\exicon{67} + A^{-1}\exicon{68}, \\
%    [[K_2]]_{\hat\basis_{\R^2,1}} & = (-A^{-5}-A^3) \exicon{66} + (-A^{-3}-A^5)\exicon{67}.\\
%    \end{split}
%\end{equation*}
%    Since $w(K_1) = w(K_2)=-1$, we have
 \begin{equation*} \label{kbsm2}
    \begin{split}
    \overline{[[K_1]]}_{\hat\basis_{\R^2,1}} & = - 2A^{2}\exicon{85} - 2A^{4}\exicon{86} - A^{6} \exicon{87}, \\
    \overline{[[K_2]]}_{\hat\basis_{\R^2,1}} & = - 2A^{2}\exicon{85} - 2A^{4}\exicon{86} - A^{6} \exicon{88},\\
    \end{split}
\end{equation*}
    By Proposition~\ref{prop:kbsm-vs-kbp-or} we can obtain the Kauffman bracket polynomial from the Kauffman bracket skein modules by replacements
    $$
    \exicon{85} \rightarrow \lambda_{12}\lambda_{34}, \;\;
    \exicon{86} \rightarrow \lambda_{14}\lambda_{23}, \;\;
    \exicon{87} \rightarrow (-A^2-A^{-2})\lambda_{14}\lambda_{23}, \;\;
    \exicon{88} \rightarrow (-A^2-A^{-2})\lambda_{14}\lambda_{23}.
    $$
    The normalized ordered Kauffman bracket polynomial does not distinguish the two multi-linkoids:
    $$\overline{<K_1>}_{\bullet} = \overline{<K_2>}_{\bullet} 
    = -2A^2\lambda_{12}\lambda_{34} + (A^8-A^4)\lambda_{14}\lambda_{23}$$
%    $$\overline{<K_1>}_\bullet (A, \{ \lambda_{12}  \})  = \overline{<K_2>}_\bullet (A, \{ \lambda_{12}  \}) = (-A^{-4} - A^4)\, \lambda_{12}.$$
    Observe that if we consider that the two multi-knotoids lie in the $S^2$, $K_1$ and $K_2$ are isotopic. Indeed,
    $$\overline{[[K_1]]}_{\basis_{S^2,1}} = \overline{[[K_1]]}_{\basis_{S^2,1}} 
    = - 2A^{2}\exicon{85} - 2A^{4}\exicon{86} - A^{6} \exicon{87}.$$

%\begin{equation*} \label{kbsm1}
%\begin{split}
%[[K_1]]_{\hat\basis_{\R^2,1}} & = (-A^3+A^7) \exicon{66} + A^{-3}\exicon{67} + A^{-1}\exicon{68}, \\
%[[K_2]]_{\hat\basis_{\R^2,1}} & = (-A^{-5}-A^3) \exicon{66} + (-A^{-3}-A^5)\exicon{67}.\\
%\end{split}
%\end{equation*}
%Since $w(K_1) = w(K_2)=-1$, we have
%\begin{equation*} \label{kbsm2}
%\begin{split}
%\overline{[[K_1]]}_{\hat\basis_{\R^2,1}} & = (1-A^4) \exicon{66} - A^{-6}\exicon{67} - A^{-4}\exicon{68}, \\
%\overline{[[K_2]]}_{\hat\basis_{\R^2,1}} & = (A^{-8}+1) \exicon{66} + (A^{-6}+A^2)\exicon{67}.\\
%\end{split}
%\end{equation*}
%By Proposition~\ref{prop:kbsm-vs-kbp-or} we can obtain the Kauffman bracket polynomial from the Kauffman bracket skein modules by replacements
%$$\exicon{66} \rightarrow \lambda_{12}, \qquad \exicon{67} \rightarrow (-A^2-A^{-2})\lambda_{12} \quad \text{and} \quad  \exicon{68} \rightarrow (-A^2-A^{-2})^2\lambda_{12}.$$
%The normalized ordered Kauffman bracket polynomial does not distinguish the two multi-linkoids:
%$$\overline{<K_1>}_\bullet (A, \{ \lambda_{12}  \})  = \overline{<K_2>}_\bullet (A, \{ \lambda_{12}  \}) = (-A^{-4} - A^4)\, \lambda_{12}.$$
%Observe that if we consider that the two multi-knotoids are in the 2-sphere, $L_1$ and $L_2$ are isotopic. Indeed,
%$$\overline{[[K_1]]}_{\basis_{S^2,1}} = \overline{[[K_1]]}_{\basis_{S^2,1}} = (A^{-1}+A^7)  \exicon{66}.$$
\end{example}

%{\color{red} In the case $n=1$ We should compare to the bracket polynomial in \cite{Turaev2012} (page 17), but his invariant is weaker, since it only counts the components, not taking into account the isotopy classes.}

%{\color{red} I do not quite understand Turaev's remark page 16, 2nd paragraph concerning the Kauffman bracket.}

\section{Spatial graphs and multi-linkoids}\label{sec:spatial}

Similar to knotoids in $\mathbb{R}^2$, 
we can consider the topological setting of multi-linkoids in $\mathbb{R}^2$. A planar multi-linkoid can be uniquely lifted in the 3-space as a relative link in $\mathbb{R}^3 \setminus \{ \text{2$n$ parallel lines} \}.$ See Figure~\ref{fig:lines}. 
By contracting these lines at $\pm\infty$, they can be presented as spacial cases of generalized $\Theta$-graphs, which we will be the topic of this section.

\begin{figure}[ht]
     \centering
      \includegraphics[page=89]{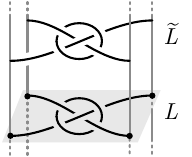}  
              \caption{A planar multi-linkoid $L$ and its lift $\tilde L$ in $\mathbb{R}^3$ as a relative knot.}
        \label{fig:lines}
\end{figure}

\begin{definition}\normalfont
A {\it spatial graph} is a presentation of a graph $G$ in $S^3$,  i.e.\ an embedding $f: G \hookrightarrow S^3$. A spatial graph diagram is a regular projection of a graph on $S^2 \subset S^3$
endowed under or over information.
%together with information of overcrossing and undercrossings of the arcs.
\end{definition}

\begin{theorem}[Kauffman~\cite{Kauffman1989}] 
Two diagrams of spatial graphs are ambient isotopic if and only if they are related through a finite sequence of moves R-0, R-I, R-II, R-III involving the edges (Figure~\ref{fig:reid}) and moves R-IV and R-V involving the vertices (Figure~\ref{fig:vert}).
\end{theorem}

\begin{figure}[ht]
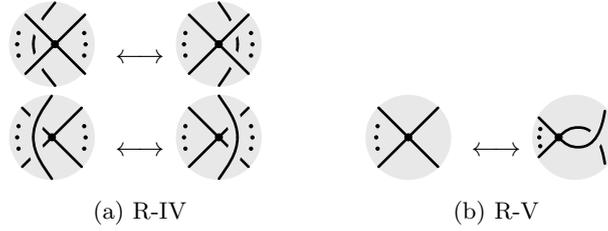

     \centering
     \begin{subfigure}[b]{0.30\textwidth}
         \centering
         \reid{52} $\longleftrightarrow$ \reid{53} \qquad \reid{54} $\longleftrightarrow$ \reid{55}
         \caption{R-IV}
         \label{fig:r4}
     \end{subfigure}
     \begin{subfigure}[b]{0.30\textwidth}
         \centering
         \reid{56} $\longleftrightarrow$ \reid{57}
         \caption{R-V}
         \label{fig:r5}
     \end{subfigure}

        \caption{Local Reidemeister moves for spatial graphs involving a vertex. The move R-IV can have an arbitrary number of strands (including zero) on the left-hand and the right-hand side of the diagram. The move R-V can have an arbitrary number of strands on the left-hand side of the diagram.}
        \label{fig:vert}
\end{figure}

\begin{definition}\normalfont

A {\it $\Theta$-graph} is a spatial graph with two labeled vertices $v_0, v_1$ and three labelled edges $e_+, e_0, e_-$ connecting the vertices. Two $\Theta$-graphs are considered to be isomorphic if there is an orientation preserving isotopy of $S^3$ between them that preserves the labelling of the vertices and the edges.

\end{definition}

One can work with $\Theta$-graphs through their diagrams in $S^2 \subset S^3$. We assume two $\Theta$-graph diagrams represent the same $\Theta$-graph in $S^3$ if they are related to each other by a sequence of spatial graph Reidemeister moves, presented in  Figure \ref{fig:r1}.

\begin{definition}\normalfont
A $\Theta$-graph is called \emph{simple} if the edges $e_+$ and $e_-$ bound a $2$-disk.

\end{definition}
A knotoid diagram in $S^2$ can be assigned to a simple $\Theta$-graph embedded in $S^3$ by corresponding its endpoints to two vertices, and the knotoid diagram is placed as the edge labeled by $e_0$ of the $\Theta$-graph. This construction was presented in \cite{Turaev2012}, where it was proven that it induces a bijection between the set of isotopy classes of spherical knotoids and the isomorphism classes of simple $\Theta$-graphs. In \cite{Turaev2012}, this bijection was used to prove the prime decomposition theorem of knotoids. Now we shall generalize the concept of a $\Theta$-graph and give a correspondence between linkoids in $S^2$ and the generalized $\Theta$-graphs. In this way, we extract invariants for linkoids from spatial graph invariants.

\begin{definition}\normalfont

A \emph{generalized} $\Theta$-graph is a connected graph embedded in $S^3$ with an even number, say $2n$, $n \in \mathbb{N}$, of trivalent vertices, labeled $v_i$ and $w_i$ where $i \in \{1,\ldots,n\}$,  and exactly two vertices $v_{\infty}$, $v_{-\infty}$ that are located at $N(0,0,0,1)\in S^3$ and $S(0,0,0,-1) \in S^3$, respectively.
The edge set $E(G)$ consists of edges $\{v_iw_i\}_{i=1}^n$, edges $\{v_i v_{\infty}\}_{i=1}^n$ and edges $\{v_i v_{-\infty}\}_{i=1}^n$ connecting $v_i$ and the vertices $v_{\infty}$ and $v_{-\infty}$, and edges 
 $\{w_i v_{\infty}\}_{i=1}^n$ and $\{w_i v_{-\infty}\}_{i=1}^n$ connecting $w_i$ and the points $v_{\infty}$ and $v_{-\infty}$.
  Note that each vertex of a pair $(v_i, w_i)$ is adjacent to both $v_{\infty}$ and  $v_{-\infty}$.
  
%{\color{gray}  
% A \emph{generalized} $\Theta$-graph is al connected graph embedded in $S^3$ with an even number, say $2n$, $n \in \mathbb{N}$, of trivalent vertices, labeled with $v_i, w_j$ where $i, j \in \{1,\ldots,n\}$,  and exactly two vertices that are located at $N= (0,0,0,1)\in S^3$ and $S=(0,0,0,-1) \in S^3$ and are labeled with $v_{-\infty}$ and $v_{\infty}$.  The edges incident to the polar vertices can be partioned into pairs so that edges in a pair are incident  to $v_i, w_i$ for some $i$ so that the degrees of $v_{-\infty}$ and $v_{\infty}$ are $2n$.  The edges that are incident to $v_{\infty}$ are labeled with $e_+$ and the edges that are incident to $v_{-\infty}$ are labeled with $e_-$.  It follows that each vertex of a pair $(v_i, w_i)$ is incident to exactly one $e_+$ and one $e_-$. 
%} 

\end{definition}

%{\color{gray}  
%\begin{definition}\normalfont
%A generalized $\Theta$-graph is simple if for every $i,j \in \{1,\ldots,n\}$, the $e_+$-labeled edges  and $e_{-}$-labeled edges incident to a pair of vertices $(v_i, w_j)$ with the related vertices form a cycle graph that can be embedded in a plane.
%\end{definition}

%\begin{definition}\normalfont
%
%\end{definition} 
%}

\begin{definition}\normalfont

A generalized $\Theta$-graph is simple if for every $i,j \in \{1,\ldots,n\}$, the subgraph induced by edges $v_i v_{\infty}$, $v_i v_{-\infty}$, $w_j v_{\infty}$, and $w_j v_{-\infty}$  
is a cycle $C_4$ that can be embedded in a plane. 
 
%{\color{gray}
%Two generalized $\Theta$-graphs $\Theta_1, \Theta_2$ are considered to be \emph{isomorphic} if there is an isotopy of $S^3$ taking $\Theta_1$ to $\Theta_2$ by preserving the labeling of the vertices.
%}
\end{definition}

Let $L$ be a multi-linkoid diagram in $S^2$ with $2n$ endpoints. We assign to $L$ to a generalized $\Theta$-graph in embedded in $S^3$, which we denote by $\Theta(L)$.
The graph $\Theta(L)$ is constructed as follows. We label each component of $L$ by $i \in \{1,\ldots,n\}$, and denote by $v_i$ the tail and by $w_i$ and the head of component $i$,
such that component $i$ corresponds to an edge $v_iw_i$ in $\Theta(L)$. We also add the two vertices $v_{\infty}$ and $v_{-\infty}$ located at  $N(0,0,0,1)\in S^3$ and $S(0,0,0,-1) \in S^3$, respectively. 

Next, we add to the edge set of $\Theta(L)$ the edges  $v_i v_{\infty}$, $v_i v_{-\infty}$, $w_j v_{\infty}$, and $w_j v_{-\infty}$, such that they form a graph $C_4$ that bounds a disk
and the arcs $v_i v_{\infty}$ and $w_i v_{\infty}$ form only over\textnormal{-}crossings with the rest of the diagram and arcs $v_i v_{-\infty}$ and $w_i v_{-\infty}$ form only under\textnormal{-}crossings with the rest of the diagram.
%{\color{red} Is this enough details to be well-defined? Should arcs connecting $v_{\infty}$ be over arcs and arcs connecting $v_{-\infty}$ underarcs?}

The assigned graph $\Theta(L)$ is a simple generalized $\Theta$-graph that contains $2n+2$ vertices, $2n$ of them are of degree $3$, and the two vertices at the poles are of degree $2n$.  
%Let us name the resulting graph as the \emph{generalized} $\Theta$-graph with four vertices.

%{\color{gray}
%Given a linkoid diagram $L$ in $S^2$ with $2n$ ends, where $n \in \mathbb{N}$. We assign $L$ to a graph embedded in $S^3$ by sending each endpoint to a labeled vertex and each knotoid component to a labeled edge connecting the vertices corresponding to their endpoints. Each pair of vertices corresponding to the two endpoints of a knotoid component is also connected by two trivial edges that bound a disk. Each pair of these edges are labelled by $+$ and $-$. The edges labeled $+$ are considered to intersect at a vertex at the North Pole, labeled $v_{\infty}$ and the edges labeled $-$ are considered to intersect at a vertex at the South Pole of $S^3$, labeled $v_{-\infty}$. The assigned graph clearly contains $2n+2$ vertices, $2n$ of them are of degree $3$, and the two vertices at the poles are of degree $2n$. 
% Let us name the resulting graph as the \emph{generalized} $\Theta$-graph with four vertices.
%}
\begin{figure}[H]
\centering
\begin{tabular}{cc}
\includegraphics[scale=1.0, page=79]{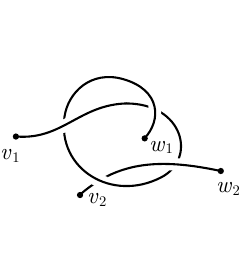} \qquad & \qquad
\includegraphics[scale=1.0, page=58]{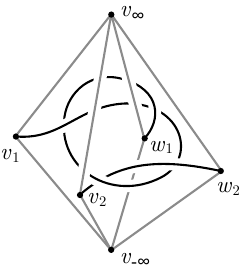} \\
$L$ & $L(\Theta)$
\end{tabular}

\caption{A linkoid and the corresponding generalized $\Theta$-graph.}

\end{figure}

\begin{theorem}
There is a one-to-one correspondence between the set of all multi-linkoids in $S^2$ and the set of simple generalized $\Theta$-graphs.
\end{theorem}

\begin{proof}

It is clear that any of the R-I, R-II, R-III moves takes place locally away from the endpoints of a multi-linkoid diagram, and they transform to either one of the R-I, R-II, R-III moves of the spatial graph, or a combination of them, in case when the arcs, connecting vertices labeled by $v_i, w_j$ to the infinity vertices, conflict with the local move region on the corresponding generalized $\theta$-graph diagram. The isotopy moves that displace endpoints of a multi-linkoid diagram transform as a combination of spatial Reidemeister moves including R-IV and R-V-moves on the corresponding generalized $\theta$-graph diagram. See Figure \ref{fig:vertexmoves} for some of the instances.

\begin{figure}[H]
\centering

\includegraphics[scale=0.16]{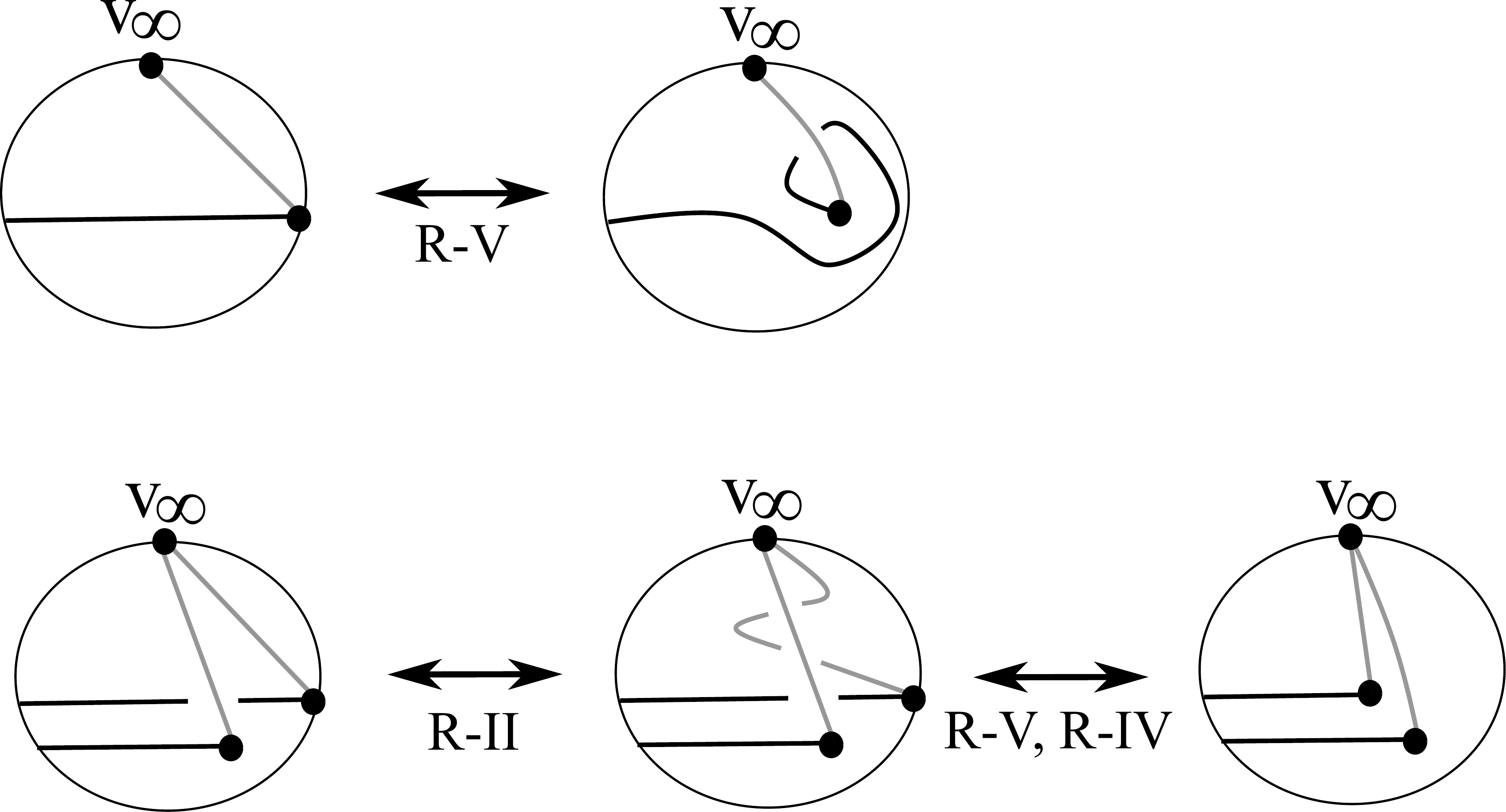}
\caption{Moves involving vertices}
\label{fig:vertexmoves}

\end{figure}

\end{proof}

\section{A colored version of Kauffman's $T$ invariant}\label{sec:T}

In~\cite{Kauffman1989} an invariant $T$ of spatial graphs in $S^3$ was introduced. For convenience, we repeat the construction.

Let $G$ be a spatial graph (a graph embedded in $S^3$ or $\mathbb{R}^3$) and let $v$ be a vertex of $G$ of degree $d$. We associate to $(G,v)$ a set consisting of $\frac{d(d-1)}{2}$ spatial graphs obtained by 
connecting two pairs of edges at this vertex and leaving all other edges as free ends (leafs) as in Figure~\ref{fig:repl}. We call such an operation a \emph{local replacement}.

\begin{figure}[H]
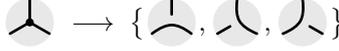

\centering
$$\icon{44} \;\longrightarrow\; \big\{ \icon{45},  \icon{46},  \icon{47} \big\}$$
\caption{Local replacements at a vertex.}\label{fig:repl}
\end{figure}

Let $T(G)$ be the set of local replacements made on every vertex of $G$, where, in addition, we remove all non-closed curves.
% More precisely,...

Clearly, $T(G)$ is a set of links in $S^3.$

\begin{theorem}[\cite{Kauffman1989}]
Let $G$ be a spatial graph. The collection $\Tc(G)$, taken up to ambient isotopy, is a topological invariant of $G$.
\end{theorem}

We will extend $T$ to an invariant of edge-colored graphs, $\Tc$, which we will use as an invariant of generalized $\Theta$-graphs.

Let $(G,C)$ be an edge-colored spatial graph, i.e.\ a spatial graph $G$ equipped with a coloring function $c: E(G) \rightarrow C$ for a set of colors $C$.
We define a colored local replacement as a local replacement of uncolored graphs, where, in addition, we color the new arcs by a subset $\gamma \subseteq C$, 
such that $\gamma$ consists of colors of the edges in the preimage of the replacement as in Figure~\ref{fig:crepl}.

\begin{figure}[H]
\centering
$$\bigicon{48}  \;\longrightarrow\;  \Big\{ \bigicon{49},  \bigicon{51},  \bigicon{50} \Big\}$$
\caption{Colored local replacements at a vertex}\label{fig:crepl}
\end{figure}

Now $\Tc(G)$ is the set of colored local replacements made on every vertex of $G$, where we remove all non-closed curves (unknotted arcs).
Clearly, the value of $\Tc(G)$ is a set of colored links with colors in the power set $\mathcal P(C)$. Such sets can be distinguished using any colored link invariant, for example, 
the multivariate Alexander polynomial~\cite{Torres1953, Morton1999} or the colored Jones-type polynomial~\cite{Aicardi2016}.

\begin{theorem}
Let $(G,c)$ be an edge-colored spatial graph, the collection of colored links $\Tc(G)$, taken up to ambient isotopy, is a topological invariant of $(G,c)$.
\end{theorem}

\begin{proof}
It is easy to check that Reidemeister moves R-I -- R-IV do not change the isotopy types of elements in $\Tc(G)$ and the move R-V permutes the elements (compare Figures~\ref{fig:crepl} and~\ref{fig:replrv}).
\end{proof}

\begin{figure}[H]
\centering
$$\raisebox{-.2em}{\bigicon{75}} \;\longrightarrow\; \bigg\{ \raisebox{-.2em}{\bigicon{76}},  \raisebox{-.2em}{\bigicon{77}},  \raisebox{-.2em}{\bigicon{78}} \bigg\} = \Big\{ \bigicon{49},  \bigicon{50},  \bigicon{51} \Big\}$$
\caption{Colored local replacements at a vertex.}\label{fig:replrv}
\end{figure}

\noindent {\bf Example.}
\begin{figure}
\begin{equation*}
\begin{split}
A &= \raisebox{-2.2em}{\includegraphics[page=59]{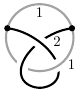}}  \qquad \longrightarrow \qquad 
T(A) = \bigg\{ \raisebox{-2.2em}{\includegraphics[page=72]{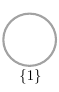}}, 
\raisebox{-2.2em}{\includegraphics[page=71]{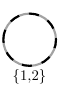}}, 
\raisebox{-2.2em}{\includegraphics[page=74]{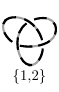}} \bigg\}
\\
B &= \raisebox{-2.2em}{\includegraphics[page=69]{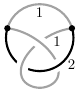}}  \qquad \longrightarrow \qquad 
T(B) = \bigg\{ \raisebox{-2.2em}{\includegraphics[page=72]{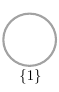}}, 
\raisebox{-2.2em}{\includegraphics[page=71]{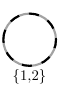}}, 
\raisebox{-2.2em}{\includegraphics[page=74]{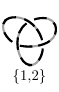}} \bigg\}
\\
C &= \raisebox{-2.2em}{\includegraphics[page=70]{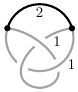}}  \qquad \longrightarrow \qquad 
T(C) = \bigg\{ 
\raisebox{-2.2em}{\includegraphics[page=71]{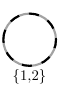}}, 
\raisebox{-2.2em}{\includegraphics[page=73]{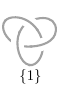}} \bigg\}
\end{split}
\end{equation*}
\caption{Three 2-colorings of  $\Theta 3_1$ and their values of $\Tc$.}\label{fig:abc}
\end{figure}
Consider the $\Theta$-curve $\Theta 3_1$ from~\cite{Moriuchi2009}, 
where we color two edges with color 0 and one edge with color 1 as depicted in Figure~\ref{fig:abc} (in the virtue of bonded knots independently introduced in~\cite{Gabrovsek2021} and~\cite{Adams2020}). We have three coloring choices of colored graphs, which we name $A$, $B$, and $C$.

It holds $\Tc(A) = \Tc(B) \neq \Tc(C)$. It is easy to verify that $A$ and $B$ are ambient isotopic and $C$ is neither ambient isotopic to $A$ nor $B$, thus $\Tc$ is able to detect the two isotopy classes.

%{\color{red} Should the definition of extended $T$ be desriptive (as it is now), or should we define it more formally? }

Given a multi-linkoid $L$, we can construct different invariants, based on $\Tc$, by modifying the coloring function on $\Theta(L)$.
The invariants vary on strength and function:

%{\color{red} We need to first define ordered (colored) linkoid and unordered linkoid.}
\begin{itemize}
\item an invariant of unordered unoriented multi-linkoids: choose $c(v_i w_i) = 0$ for the edges of $K$ and choose $c(v_i v_{\pm \infty})=c(w_i v_{\pm \infty}) = 1$ for the edges adjacent to points $v_\infty$ and $v_{-\infty}$,
\item an invariant of unordered unoriented multi-linkoids (stronger version): choose $c(v_i w_i) = 0$ for the edges of $K$ and choose $c(v_i v_{\infty})=c(w_i v_{\infty}) = 1$ and $c(v_i v_{-\infty})=c(w_i v_{-\infty}) = -1$,
\item an invariant of unordered oriented multi-linkoids: choose $c(v_i w_i) = 0$ for the edges of $K$ and choose $c(v_i v_{\pm\infty})=1$ and $c(w_i v_{\pm\infty})=2$,
\item an invariant of unordered oriented multi-linkoids (stronger version): choose $c(v_i w_i) = 0$ for the edges of $K$ and choose $c(v_i v_{\infty})=1$, $c(v_i v_{-\infty})=-1$, $c(w_i v_{\infty})=2$, $c(w_i v_{-\infty})=-2$,
\item an invariant of ordered multi-linkoids: choose $c(v_i w_i) = i$ for the edges of $K$ and choose $c(v_i v_{\pm \infty})=(i,1)$ and $c(w_i v_{\pm\infty})=(i,2)$,
\item an invariant of ordered multi-linkoids (stronger version): choose $c(v_i w_i) = i$ for the edges of $K$ and choose $c(v_i v_{\infty})=(i,1)$, $c(w_i v_{\infty})=(i,2)$, $c(w_i v_{-\infty})=(i,-1)$ and $=c(w_i v_{-\infty})=(i,-2)$.
\end{itemize}

\section*{Acknowledgements}
The first author was supported by the Slovenian Research Agency program P1-0292.\\[.5em]
\noindent This manuscript has no associate data.

\bibliographystyle{abbrv}
\bibliography{biblio}

\end{document}